\documentclass{amsart}
\usepackage{amscd,enumerate}
\usepackage{amsfonts}
\usepackage{amssymb}
\usepackage{amsmath}

\theoremstyle{plain}
\newtheorem{theorem}{Theorem}[section]

\newtheorem{corollary}[theorem]{Corollary}

\newtheorem{lemma}[theorem]{Lemma}

\newtheorem{proposition}[theorem]{Proposition}

\begin{document}
\title{Similarity for zero-square matrices}
\author{Grigore C\u{a}lug\u{a}reanu}

\begin{abstract}
We show that an $n\times n$ zero-square matrix over a commutative unital
ring $R$ is similar to a multiple of $E_{1n}$ if $R$ is a B\'{e}zout domain
and $n=2$, $3$, but there are zero-square matrices which are not similar to
any multiple of $E_{1n}$ whenever $n\geq 4$, over any commutative unital
ring. As a consequence, for $n=2$, $3$ such matrices have stable range one.
\end{abstract}

\maketitle

\section{Introduction}

An integral domain is a \textsl{GCD} domain if every pair $a$, $b$ of
nonzero elements has a greatest common divisor, denoted by $\mathrm{gcd}%
(a,b) $ and a \textsl{B\'{e}zout} domain if $\gcd (a,b)$ is a linear
combination of $a$ and $b$. GCD domains include unique factorization
domains, B\'{e}zout domains and valuation domains. If $\gcd (a,b)=1$ we say
that $a$ and $b$ are \textsl{coprime}.

It is not hard to prove that \emph{every zero-square }$2\times 2$\emph{\
matrix over a B\'{e}zout domain }$R$\emph{\ is similar to }$rE_{12}$\emph{,
for some }$r\in R$ (see \cite{CZ}).

The aim of this paper is to extend the above result for zero-square $3\times
3$ matrices over B\'{e}zout domains and to show that the property cannot be
extended for $n\times n$ zero-square matrices if $n\geq 4$. That is, we
prove the following

\textbf{Theorem}. \emph{Let }$R$\emph{\ be a B\'{e}zout domain. Every
zero-square matrix of }$\mathbb{M}_{3}(R)$\emph{\ is similar to }$rE_{13}$%
\emph{\ for some }$r\in R$\emph{.}

\textbf{Theorem}. \emph{Over any commutative ring and for every }$n\geq 4$%
\emph{, there are zero-square }$n\times n$\emph{\ matrices which are not
similar to multiples of }$E_{1n}$\emph{.}

In our extension we have to solve a special type of completion problem: two
unimodular $3$-rows are given with some additional properties and we are
searching for a completion to an invertible $3\times 3$ matrix.

In Section 2, general results on zero-square $n\times n$ matrices are proved
together with second theorem above.

For the sake of completeness, Section 3 covers the zero-square $2\times 2$
case. In Section 4 we prove the first theorem above, that is, we settle the $%
3\times 3$ zero-square case. Since multiples of $E_{ij}$ are known to have
stable range one, a consequence of our results is that zero-square $2\times
2 $ and $3\times 3$ matrices over any B\'{e}zout domain (in particular over
the integers) have stable range one.

$E_{ij}$ denotes the $n\times n$ matrix with all entries zero excepting the $%
(i,j)$ entry which is $1$. By $0_{n}$ we denote the zero $n\times n$ matrix.
For a square matrix $A$ over a \emph{commutative} ring $R$, the determinant
and trace of $A$ are denoted by $\mathrm{det}(A)$ and $\mathrm{Tr}(A)$,
respectively. For a matrix $A$, $\gcd (A)$ denotes the greatest common
divisor of all the entries of $A$. For a unital ring $R$, $U(R)$ denotes the
set of all the units of $R$.

\section{Zero-square $n\times n$ matrices}

In order to describe the zero-square $n\times n$ matrices over commutative
(unital) rings or over integral domains, denote by $T_{ab}^{cd}$ the $%
2\times 2$ minor on the rows $a$ and $b$ and on the columns $c$ and $d$. A
simple computation of $\mathrm{row}_{i}(T)\cdot \mathrm{col}_{j}(T)$ for $%
i\neq j$ or $i=j$ gives

\begin{proposition}
\label{squ}Let $T=[t_{ij}]_{1\leq i,j\leq n}$ be an $n\times n$ matrix over
a commutative ring $R$ and let $t_{ij}^{(2)}$ be the entries of $T^{2}$.
Then 
\begin{equation*}
\begin{array}{cccc}
t_{ij}^{(2)} & = & t_{ij}\mathrm{Tr}(T)+\sum_{k\in
\{1,...,n\}-\{i,j\}}T_{ik}^{kj} & i\neq j \\ 
t_{ii}^{(2)} & = & t_{ii}\mathrm{Tr}%
(T)+T_{i1}^{1i}+...+T_{i,i-1}^{i-1,i}+T_{i,i+1}^{i+1,i}+...+T_{in}^{ni} & i=j%
\end{array}%
.
\end{equation*}
\end{proposition}

First recall that the \textsl{rank} of a (not necessarily square) matrix $A$
(denoted $\mathrm{rk}(A)$) can be defined \emph{over any commutative ring} $%
R $, using the annihilators of the ideals $I_{t}(A)$ generated by the $%
t\times t$ minors of $A$ (see e. g. \cite{bro}). In particular, $\mathrm{rk}%
(A)=1$ if all $2\times 2$ minors are zero and and these two condition are
equivalent over integral domains. Then it can be shown that \emph{equivalent
matrices} (so, in particular, similar matrices) \emph{have the same rank}
(see \cite{bro}, \textbf{4.11}).

Therefore

\begin{corollary}
\label{rt}Let $T$ be an $n\times n$ matrix over any commutative ring. If all 
$2\times 2$ minors of $T$ are zero and $\mathrm{Tr}(T)=0$ then $T^{2}=0_{n}$.
\end{corollary}

\textbf{Remark}. Over any integral domain a (well-known) converse also
holds: \emph{If }$T^{2}=0_{n}$\emph{\ then }$\det (T)=\mathrm{Tr}(T)=0$.

\bigskip

Over any integral domain, in order to have a characterization of form 
\begin{equation*}
T^{2}=0_{n}\text{ if and only if }\mathrm{rk}(T)=1\text{ and }\mathrm{Tr}%
(T)=0,
\end{equation*}%
the only remaining implication is that, $T^{2}=0_{n}$ and $\mathrm{Tr}(T)=0$
imply $\mathrm{rk}(T)=1$ (i.e. all $2\times 2$ minors of $T$ vanish).

In what follows we show that this implication holds over a commutative ring
for $n=3$ if $2$ is not a zero divisor, but fails for any $n\geq 4$.

\begin{theorem}
\label{2-min}Let $R$ be a commutative unital ring such that $2$ is not a
zero divisor and let $T\in \mathbb{M}_{3}(R)$ with $\mathrm{Tr}(T)=0$. Then $%
T^{2}=0_{3}$ if and only if all $2\times 2$ minors of $T$ equal zero.
\end{theorem}

\begin{proof}
To avoid too many indexes and emphasize the diagonal elements (i.e. the zero
trace) we write $T=\left[ 
\begin{array}{ccc}
x & a & c \\ 
b & y & e \\ 
d & f & -x-y%
\end{array}%
\right] $.

If $\mathrm{Tr}(T)=0$, the condition $T^{2}=0_{3}$ is \emph{equivalent} to
the following nine LHS equalities 
\begin{equation*}
\begin{array}{c}
x^{2}+ab+cd=0\text{ \ \ \ (1)} \\ 
a(x+y)+cf=0\text{ \ \ \ (2)} \\ 
ae=cy\text{ \ \ \ (3)} \\ 
b(x+y)+de=0\text{ \ \ \ (4)} \\ 
y^{2}+ab+ef=0\text{ \ \ \ (5)} \\ 
bc=ex\text{ \ \ \ (6)} \\ 
bf=dy\text{ \ \ \ (7)} \\ 
ad=fx\text{ \ \ \ (8)} \\ 
(x+y)^{2}+cd+ef=0\text{ \ \ \ (9)}%
\end{array}%
\begin{array}{c}
\\ 
T_{13}^{23}=0 \\ 
T_{12}^{23}=0 \\ 
T_{23}^{13}=0 \\ 
\\ 
T_{12}^{13}=0 \\ 
T_{23}^{12}=0 \\ 
T_{13}^{12}=0 \\ 
\end{array}%
.
\end{equation*}%
The two terms equalities (i.e., (3), (6), (7), (8)) are \emph{equivalent} to
the vanishing of four $2\times 2$ minors. Just look at the RHS column of
vanishing minors. Further, two other equalities, namely, (2) and (4), are 
\emph{equivalent} to the vanishing of another two minors.

Thus, this equivalently covers the six \emph{off diagonal} $2\times 2$
minors. What remains are the vanishing of the three $2\times 2$ \emph{%
diagonal} minors.

From $x^{2}+ab+cd=0$, $y^{2}+ab+ef=0$ and $(x+y)^{2}+cd+ef=0$ we get (since $%
2$ is not a zero divisor) $xy=ab$, and so another zero $2\times 2$ minor.
Finally using $x^{2}+ab+cd=0$, $y^{2}+ab+ef=0$ and $xy=ab$, we get the last
two zero $2\times 2$ diagonal minors: $x(x+y)+cd=0$ and $y(x+y)+ef=0$.

The converse was settled in the general $n\times n$ case in Corollary \ref%
{rt}.
\end{proof}

\textbf{Remark}. The hypothesis "$2$\emph{\ is not a zero divisor}" is
essential for the vanishing of the three diagonal $2\times 2$ minors (over
any commutative ring). Consider $R=\mathbb{Z}_{2}[X,Y]/I$ for $%
I:=(X^{2},Y^{2})$ and the diagonal matrix over $R$, $T=\left[ 
\begin{array}{ccc}
X+I & 0 & 0 \\ 
0 & Y+I & 0 \\ 
0 & 0 & X+Y+I%
\end{array}%
\right] $. Then $T^{2}=0_{3}$, $\mathrm{Tr}(T)=0$, but the diagonal minors
are not zero. Clearly, $2$ is a zero divisor in $R$.

\bigskip

Before dealing with the $3\times 3$ matrices case, here is an example of $%
4\times 4$ zero-square matrix (over any commutative unital ring) with zero
trace and rank 2.

\textbf{Example}. $C_{4}=\left[ 
\begin{array}{cccc}
0 & 0 & 1 & 1 \\ 
0 & 0 & 1 & 1 \\ 
-1 & 1 & 0 & 0 \\ 
1 & -1 & 0 & 0%
\end{array}%
\right] ^{2}=0_{4}$, has zero trace but many not zero $2\times 2$ minors
(e.g. $\left[ 
\begin{array}{cc}
0 & 1 \\ 
1 & 0%
\end{array}%
\right] $, in the center).

Hence $T^{2}=0_{4}$ does not generally imply $\mathrm{rk}(T)=1$. Adding to
this example as many zero rows and columns as necessary, $T^{2}=0_{n}$ does
not generally imply $\mathrm{rk}(T)=1$, for any $n\geq 5$.

Since nonzero multiples of $E_{1n}$ have rank 1, and similar matrices have
the same rank, we obtain

\begin{theorem}
Over any commutative unital ring and for every $n\geq 4$, there are $n\times
n$ zero-square matrices which are not similar to any multiple of $E_{1n}$.
\end{theorem}

\section{The zero-square $3\times 3$ case}

The following lemma and proposition will be useful for the extension from $%
2\times 2$ to zero-square $3\times 3$ matrices.

\begin{lemma}
\label{1}Let $a,b,c,a^{\prime },b^{\prime },c^{\prime }\in R$, a GCD domain.
If $ab^{\prime }=a^{\prime }b$, $ac^{\prime }=a^{\prime }c$, $bc^{\prime
}=b^{\prime }c$ and the rows $\left[ 
\begin{array}{ccc}
a & b & c%
\end{array}%
\right] $ and $\left[ 
\begin{array}{ccc}
a^{\prime } & b^{\prime } & c^{\prime }%
\end{array}%
\right] $ are unimodular then the pairs $a,a^{\prime }$, $b,b^{\prime }$ and 
$c,c^{\prime }$ are associated. Moreover, there exists a unit $u\in U(R)$
such that $\left[ 
\begin{array}{ccc}
a^{\prime } & b^{\prime } & c^{\prime }%
\end{array}%
\right] =\left[ 
\begin{array}{ccc}
a & b & c%
\end{array}%
\right] u$.
\end{lemma}

\begin{proof}
Denote $\delta =\gcd (a,b)$ with $a=\delta a_{1}$, $b=\delta b_{1}$ and $%
\delta ^{\prime }=\gcd (a^{\prime },b^{\prime })$ and $a^{\prime }=\delta
^{\prime }a_{1}^{\prime }$, $b^{\prime }=\delta ^{\prime }b_{1}^{\prime }$.
From $ab^{\prime }=a^{\prime }b$ cancelling $\delta \delta ^{\prime }$ we
obtain $a_{1}b_{1}^{\prime }=a_{1}^{\prime }b_{1}$. Since $a_{1},b_{1}$ are
coprime, it follows $a_{1}\mid a_{1}^{\prime }$. Symmetrically, since $%
a_{1}^{\prime },b_{1}^{\prime }$ are coprime, it follows $a_{1}^{\prime
}\mid a_{1}$, so that $a_{1},a_{1}^{\prime }$ are associates. Hence there is
a unit $u\in U(R)$ such that $a_{1}=a_{1}^{\prime }u$.

Further, notice that $\gcd (\delta ,c)=\gcd (\gcd (a,b),c)=1$ and so $\delta
,c$ are coprime. Now we use $ac^{\prime }=a^{\prime }c$, that is, $\delta
(a_{1}^{\prime }u)c^{\prime }=\delta a_{1}c^{\prime }=\delta ^{\prime
}a_{1}^{\prime }c$. Cancelling $a_{1}^{\prime }$ we get $\delta uc^{\prime
}=\delta ^{\prime }c$ and since $\delta ,c$ are coprime, $\delta \mid \delta
^{\prime }$. Symmetrically, $\delta ^{\prime }\mid \delta $ and so $\delta
,\delta ^{\prime }$ are also associates. Therefore $a=\delta a_{1}$ and $%
a^{\prime }=\delta ^{\prime }a_{1}^{\prime }$ are associates.

In a similar way, it follows that $b,b^{\prime }$ and $c,c^{\prime }$ are
associates, respectively.

Finally, suppose $a^{\prime }=au$, $b^{\prime }=bv$ and $c^{\prime }=cw$ for
some $u,v,w\in U(R)$. From $ab^{\prime }=a^{\prime }b$ we get $abv=aub$, so $%
v=u$. Analogously, $w=v$ and so $w=v=u$, as claimed.
\end{proof}

\textbf{Remark}. The second hypothesis of the lemma can be stated as a
matrix rank: $\mathrm{rk}\left[ 
\begin{array}{ccc}
a & b & c \\ 
a^{\prime } & b^{\prime } & c^{\prime }%
\end{array}%
\right] =1$.

\begin{proposition}
\label{2}Let $R$ be a GCD domain. If $\mathrm{rk}\left[ 
\begin{array}{ccc}
a & b & c \\ 
a^{\prime } & b^{\prime } & c^{\prime }%
\end{array}%
\right] =1$ (i.e., $ab^{\prime }=a^{\prime }b$, $ac^{\prime }=a^{\prime }c$, 
$bc^{\prime }=b^{\prime }c$), $\delta =\gcd (a,b,c)$, $\lambda =\gcd
(a^{\prime },b^{\prime },c^{\prime })$ and $a=\delta a_{1}$, $b=\delta b_{1}$%
, $c=\delta c_{1}$, $a^{\prime }=\lambda a_{1}^{\prime }$, $b^{\prime
}=\lambda b_{1}^{\prime }$ and $c^{\prime }=\lambda c_{1}^{\prime }$, then $%
a_{1},b_{1},c_{1}$ and $a_{1}^{\prime },b_{1}^{\prime },c_{1}^{\prime }$ are
respectively associated (in divisibility). Moreover, $\left[ 
\begin{array}{ccc}
a_{1}^{\prime } & b_{1}^{\prime } & c_{1}^{\prime }%
\end{array}%
\right] =\left[ 
\begin{array}{ccc}
a_{1} & b_{1} & c_{1}%
\end{array}%
\right] u$ for some $u\in U(R)$.
\end{proposition}

\begin{proof}
We just use the previous lemma.
\end{proof}

In the sequel, for 3-vectors we use the well-known operations of \emph{dot}
product, \emph{cross} product and \emph{scalar triple} product.

\textbf{Definition}. The 3-vector $\mathbf{a}=(a_{1},a_{2},a_{3})\in R^{3}$
is \textsl{unimodular} iff the ideal generated by its components is the
whole ring, i.e. $I=(a_{1},a_{2},a_{3})=Ra_{1}+Ra_{2}+Ra_{3}=R$.
Equivalently, there exists $\mathbf{b}=(b_{1},b_{2},b_{3})\in R^{3}$ such
that $\mathbf{a}\cdot \mathbf{b}=1$.

More detailed, for 3 elements of a ring $a_{1},a_{2},a_{3}\in R$, the ideal
generated by these $I=(a_{1},a_{2},a_{3})=Ra_{1}+Ra_{2}+Ra_{3}$ can be the
whole ring $R$, case when $\{a_{1},a_{2},a_{3}\}$ (ideal) generates the
whole $R$, or else, it is \emph{not} the whole ring. Since by Zorn's Lemma,
every proper ideal is included in a maximal ideal, the second case can be
characterized as follows: the system $\{a_{1},a_{2},a_{3}\}$ is \emph{not}
an (ideal) generating system iff these elements (and so is the ideal these
generate) are included in a maximal ideal.

This way, a 3-vector $\mathbf{a}=(a_{1},a_{2},a_{3})\in R^{3}$ is unimodular
iff $\{a_{1},a_{2},a_{3}\}$ is not included in any maximal ideal $M$ of $R$.
Equivalently, for every maximal ideal $M$ of $R$, at least one of the $%
a_{i}\notin M$, or else, at least one of $a_{1}+M,a_{2}+M,a_{3}+M\in R/M$ is 
$\neq M$ (i.e. is not zero in $R/M$).

To simplify the writing, we denote $\mathbf{a}+M=(a_{1}+M,a_{2}+M,a_{3}+M)$,
and this can be viewed as a 3-vector in $(R/M)^{3}$. Moreover, we extend
accordingly the dot product $(\mathbf{a}+M)\cdot (\mathbf{b}+M)=\mathbf{a}%
\cdot \mathbf{b}+M\in R/M$.

\begin{proposition}
Suppose $\mathbf{a}=(a_{1},a_{2},a_{3})$, $\mathbf{b}=(b_{1},b_{2},b_{3})$, $%
\mathbf{c}=(c_{1},c_{2},c_{3})$ are unimodular 3-vectors such that $\mathbf{%
a\cdot b}=1$, $\mathbf{a\cdot c}=0$. Then the cross product $\mathbf{b}%
\times \mathbf{c}$ is also a unimodular row.
\end{proposition}

\begin{proof}
As mentioned above, it suffices to show that the 3-vector $\mathbf{b}\times 
\mathbf{c}$ (as customarily identified with the (ideal) generating system $%
\{b_{2}c_{3}-b_{3}c_{2},b_{3}c_{1}-b_{1}c_{3},b_{1}c_{2}-b_{2}c_{1}\}$) is
nonzero, modulo any maximal ideal. Since $R$ is a commutative (unital) ring,
modulo any maximal ideal $M$ of $R$, $R/M$ is a field and (with the above
notation) $\mathbf{b}+M$, $\mathbf{c}+M$ are nonzero 3-vectors in $(R/M)^{3}$
(otherwise these are not unimodular). It is easy to see that these two
vectors are linearly independent (indeed, if $(\mathbf{a}+M)\mathbf{\cdot }(%
\mathbf{b}+M)=1+M$, $(\mathbf{a}+M)\mathbf{\cdot }(\mathbf{c}+M)=M$ and $%
\mathbf{b}+M=k(\mathbf{c}+M)$ for some $k\in R$ then $1+M=M$, impossible).
Hence their cross product is (well-known to be) nonzero and the proof is
complete.
\end{proof}

\begin{proposition}
Let $R$ be a commutative ring and let $\mathbf{a}$, $\mathbf{b}$, $\mathbf{c}
$ be unimodular $3$-vectors such that $\mathbf{a}\cdot \mathbf{b}=1$ and $%
\mathbf{a\cdot c}=0$. There exists a unimodular $3$-vector $\mathbf{x}$,
also orthogonal on $\mathbf{a}$, such that $\mathbf{b}\cdot (\mathbf{x}%
\times \mathbf{c})=1$.
\end{proposition}

\begin{proof}
By the above proposition, since $\mathbf{b}\times \mathbf{c}$ (which is just
the three $2\times 2$ minors of the matrix $[\mathbf{bc}]=\left[ 
\begin{array}{ccc}
b_{1} & b_{2} & b_{3} \\ 
c_{1} & c_{2} & c_{3}%
\end{array}%
\right] $) is also unimodular, there exists a unimodular $3$-vector $\mathbf{%
x}$ such that $(\mathbf{b}\times \mathbf{c})\cdot \mathbf{x}=1$, that is, $%
\det [\mathbf{bcx}]=1$.

Hence $\mathbf{b}\cdot (\mathbf{x}\times \mathbf{c})=1$. If $\mathbf{a}\cdot 
\mathbf{x}=s$, replace $\mathbf{x}$ by $\mathbf{x}-s\mathbf{b}$ and then
this vector is also orthogonal on $\mathbf{a}$ (indeed, $\mathbf{a}\cdot (%
\mathbf{x}-s\mathbf{b})=\mathbf{a}\cdot \mathbf{x}-s(\mathbf{a}\cdot \mathbf{%
b})=s-s=0$).
\end{proof}

We are now ready to prove our main result

\begin{theorem}
\label{3}Zero-square $3\times 3$ matrices over a B\'{e}zout domain, are
similar to multiples of $E_{13}$.
\end{theorem}

\begin{proof}
Again consider $T=\left[ 
\begin{array}{ccc}
x & a & c \\ 
b & y & e \\ 
d & f & -x-y%
\end{array}%
\right] $ with $T^{2}=0_{3}$ (by Theorem \ref{2-min}, $\mathrm{rank}(T)=1$,
that is, all $2\times 2$ minors are zero).

Denote $\delta =\gcd (x,a,c)$, $\lambda =\gcd (b,y,e)$ and $\gamma =\gcd
(d,f,x+y)$ so that $x=\delta x_{1}$, $a=\delta a_{1}$, $c=\delta c_{1}$, $%
b=\lambda b_{1}$, $y=\lambda y_{1}$, $e=\lambda e_{1}$, $d=\gamma d_{1}$, $%
f=\gamma f_{1}$ and $x+y=\gamma (x_{2}+y_{2})$.

According to Proposition \ref{2}, there are units $u,v$ such that $\left[ 
\begin{array}{ccc}
b_{1} & y_{1} & e_{1}%
\end{array}%
\right] =\left[ 
\begin{array}{ccc}
x_{1} & a_{1} & c_{1}%
\end{array}%
\right] u$ and $\left[ 
\begin{array}{ccc}
d_{1} & f_{1} & -x_{2}-y_{2}%
\end{array}%
\right] =\left[ 
\begin{array}{ccc}
x_{1} & a_{1} & c_{1}%
\end{array}%
\right] v$.

Hence $T=\left[ 
\begin{array}{ccc}
\delta x_{1} & \delta a_{1} & \delta c_{1} \\ 
\lambda ux_{1} & \lambda ua_{1} & \lambda uc_{1} \\ 
\gamma vx_{1} & \gamma va_{1} & \gamma vc_{1}%
\end{array}%
\right] $ and since $\left[ 
\begin{array}{ccc}
x_{1} & a_{1} & c_{1}%
\end{array}%
\right] $ is unimodular, there are $s,t,z\in R$ and $sx_{1}+ta_{1}+zc_{1}=1$.

Note that $\mathrm{Tr}(T)=\delta x_{1}+\lambda ua_{1}+\gamma vc_{1}=0$.

Denote $r=\gcd (\delta ,\lambda ,\gamma )=\gcd (T)$ and denote $\delta
=r\delta _{1}$, $\lambda =r\lambda _{1}$, $\gamma =r\gamma _{1}$. We are
looking for an invertible matrix $U$ such that $TU=U(rE_{13})=\left[ 
\begin{array}{ccc}
0 & 0 & ru_{11} \\ 
0 & 0 & ru_{21} \\ 
0 & 0 & ru_{31}%
\end{array}%
\right] $.

Our choice for $r$ is necessary: indeed, writing $T=rUE_{13}U^{-1}$, shows
that $r$ divides all entries of $T$. Also note that, if $\det (U)=1$, every
row and every column of $U$ is unimodular.

We choose $\mathrm{col}_{3}(U)=\left[ 
\begin{array}{c}
s \\ 
t \\ 
z%
\end{array}%
\right] $. By computation

$ru_{11}=\mathrm{row}_{1}(T)\cdot \mathrm{col}_{3}(U)=\delta
(x_{1}u_{13}+a_{1}u_{23}+c_{1}u_{33})=\delta $,

$ru_{21}=\mathrm{row}_{2}(T)\cdot \mathrm{col}_{3}(U)=\lambda
u(x_{1}u_{13}+a_{1}u_{23}+c_{1}u_{33})=\lambda u$,

$ru_{31}=\mathrm{row}_{3}(T)\cdot \mathrm{col}_{3}(U)=\gamma
v(x_{1}u_{13}+a_{1}u_{23}+c_{1}u_{33})=\gamma v$, and,

$\left[ 
\begin{array}{ccc}
x_{1} & a_{1} & c_{1}%
\end{array}%
\right] \left[ 
\begin{array}{c}
u_{11} \\ 
u_{21} \\ 
u_{31}%
\end{array}%
\right] =\left[ 
\begin{array}{ccc}
x_{1} & a_{1} & c_{1}%
\end{array}%
\right] \left[ 
\begin{array}{c}
u_{12} \\ 
u_{22} \\ 
u_{32}%
\end{array}%
\right] =0$ and so

$\left[ 
\begin{array}{ccc}
x_{1} & a_{1} & c_{1}%
\end{array}%
\right] U=\left[ 
\begin{array}{ccc}
0 & 0 & 1%
\end{array}%
\right] $.

Hence, the first column of $U$ must be $\mathrm{col}_{1}(U)=\left[ 
\begin{array}{c}
u_{11} \\ 
u_{21} \\ 
u_{31}%
\end{array}%
\right] =\left[ 
\begin{array}{c}
\dfrac{\delta }{r} \\ 
\dfrac{\lambda }{r}u \\ 
\dfrac{\gamma }{r}v%
\end{array}%
\right] $. These fractions exist since $r=\gcd (\delta ,\lambda ,\gamma )$.

We indeed have $\left[ 
\begin{array}{ccc}
x_{1} & a_{1} & c_{1}%
\end{array}%
\right] \left[ 
\begin{array}{c}
\dfrac{\delta }{r} \\ 
\dfrac{\lambda }{r}u \\ 
\dfrac{\gamma }{r}v%
\end{array}%
\right] =\dfrac{1}{r}(\delta x_{1}+\lambda ua_{1}+\gamma vc_{1})=\dfrac{1}{r}%
(x+y-(x+y))=0$ (because $\mathrm{Tr}(T)=0$).

We are searching for a suitable column $\mathrm{col}_{2}(U)$ such that $U=%
\left[ 
\begin{array}{ccc}
\delta _{1} & u_{12} & s \\ 
\lambda _{1}u & u_{22} & t \\ 
\gamma _{1}v & u_{32} & z%
\end{array}%
\right] $ is invertible and $\left[ 
\begin{array}{ccc}
x_{1} & a_{1} & c_{1}%
\end{array}%
\right] U=\left[ 
\begin{array}{ccc}
0 & 0 & 1%
\end{array}%
\right] $.

Taking $\mathbf{a}=\left[ x_{1} \ a_{1} \ c_{1} \right]$, $\mathbf{b}=\left[
s \ t \ z \right] $ and $\mathbf{c}=\left[ \delta _{1} \ \lambda _{1}u \
\gamma _{1}v \right]$ the existence of $\mathbf{x}=\left[ u_{12} \ u_{22} \
u_{32} \right] $ follows (by transpose) from the previous proposition.
\end{proof}

\textbf{Example}. Now $x_{1}=6=2\cdot 3$, $a_{1}=10=2\cdot 5$, $%
c_{1}=15=3\cdot 5$, so that no two of these are coprime.

$T_{6}=\left[ 
\begin{array}{ccc}
-180 & -300 & -450 \\ 
90 & 150 & 225 \\ 
12 & 20 & 30%
\end{array}%
\right] $, $\delta =-30$, $\lambda =15$, $\gamma =2$ and so $r=1$.

The second equation is a linear Diophantine equation, $6s+10t+15z=1$. We
denote $w=3s+5t$ and solve $2w+15z=1$. It gives $w=-7+15n$, $z=1-2n$.

We choose $w=-7$ (for $n=0$) and solve $3s+5t=-7$. This gives for instance $%
s=-14$, $t=7$, so we choose also $z=1$ and $U=\left[ 
\begin{array}{ccc}
-30 & u_{12} & -14 \\ 
15 & u_{22} & 7 \\ 
2 & u_{32} & 1%
\end{array}%
\right] $.

Now the first equation is 
\begin{equation*}
-u_{12}\left\vert 
\begin{array}{cc}
15 & 7 \\ 
2 & 1%
\end{array}%
\right\vert +u_{22}\left\vert 
\begin{array}{cc}
-30 & -14 \\ 
2 & 1%
\end{array}%
\right\vert -u_{32}\left\vert 
\begin{array}{cc}
-30 & -14 \\ 
15 & 7%
\end{array}%
\right\vert =1,
\end{equation*}%
that is, $-u_{12}-2u_{22}=1$.

Hence $2u_{22}=-1-u_{12}$ and so $6u_{12}-5-5u_{12}+15u_{32}=0$ or $%
u_{12}+15u_{32}=5$. We can choose $u_{12}=5$, $u_{32}=0$ and so $u_{22}=-3$.

Indeed $\left[ 
\begin{array}{ccc}
-180 & -300 & -450 \\ 
90 & 150 & 225 \\ 
12 & 20 & 30%
\end{array}%
\right] \left[ 
\begin{array}{ccc}
-30 & 5 & -14 \\ 
15 & -3 & 7 \\ 
2 & 0 & 1%
\end{array}%
\right] =\left[ 
\begin{array}{ccc}
0 & 0 & -30 \\ 
0 & 0 & 15 \\ 
0 & 0 & 2%
\end{array}%
\right] $, as desired.

\end{document}